\documentclass[12pt]{amsart}
\usepackage{graphicx}
\usepackage{pict2e}
\usepackage{amssymb}
\usepackage{amsmath}
\usepackage{latexsym}

\def\R{\mathbb R}
\def\CC{\overline{\mathbb C}}
\def\C{\mathbb C}
\def\N{\mathbb N}
\def\card{\operatorname{card}}

\def\deg{\operatorname{deg}}
\def\dim{\operatorname{dim}}
\def\diam{\operatorname{diam}}
\newtheorem{lemma}{Lemma}
\newtheorem{theorem}{Theorem}
\theoremstyle{definition}

\theoremstyle{remark}
\newtheorem*{remark}{Remark}

\numberwithin{equation}{section}
\begin{document}

\title{The size of Julia sets of quasiregular maps}

\author{Walter Bergweiler}
\address{Mathematisches Seminar,
Christian--Albrechts--Universit\"at zu Kiel,
Lude\-wig--Meyn--Str.~4,
D--24098 Kiel,
Germany}
\email{bergweiler@math.uni-kiel.de}
\thanks{Supported by the ESF Networking Programme HCAA}
\begin{abstract}
Sun Daochun and Yang Lo have shown that many 
results of the Fatou-Julia iteration theory of rational functions 
extend to  quasiregular self-maps of the Riemann sphere 
for which the degree exceeds the dilatation. 
We show that in this context, in contrast to the case of rational functions,
the Julia set may have Hausdorff dimension zero.
On the other hand, we exhibit a gauge function depending on the degree
and the dilatation such
that the Hausdorff measure with respect to this gauge function
is always positive, but may be finite.
\end{abstract}
\maketitle

\section{Introduction}
Sun Daochun and Yang Lo \cite{Sun99, Sun00, Sun01} have extended many results
of the Fatou-Julia iteration theory of rational functions to quasiregular
maps $f\colon \CC\to\CC$ for which the degree $\deg(f)$ exceeds the dilatation
$K(f)$.
Here $\CC=\C\cup\{\infty\}$ is the Riemann sphere.
The key idea is to define the Julia set $J(f)$ of such a map $f$ not via
non-normality but as the set of all points $z$ such that for all
neighborhoods $U$ of $z$ the forward orbit
\[
O^+_f(U)=\bigcup_{k\ge 0} f^k(U)
\]
misses at most two points of the sphere; that is,
\[
J(f)=\left\{z\in\CC\colon 
\card\!\left(\CC\backslash O^+_f(U)\right)\leq 2 \text{ for all 
neighborhoods } U \text{ of } z\right\}.
\]
Here $\card X$ denotes the cardinality of a set~$X$.

For example, Sun and Yang \cite[Theorem~9]{Sun00} proved
that if 
$z\in J(f)$, then the backward orbit
\[
O^-_f(z)=\bigcup_{k\ge 0} f^{-k}(z)
=\bigcup_{k\ge 0} \{\zeta\in\CC\colon f^{k}(\zeta)=z\}
\]
is dense in $J(f)$; that is, $J(f)=\overline{O^-_f(z)}$.
Also, the exceptional set 
\[
E(f)=\left\{z\in\CC\colon O^-_f(z) \text{ is finite} \right\}
\]
contains at most two points, and we have $J(f)\cap E(f)=\emptyset$
and $J(f)\subset\overline{O^-_f(z)}$ for all $z\in\CC\backslash E(f)$.
Many other results of complex dynamics have been extended by
Sun and Yang to quasiregular self-maps of the Riemann sphere $\CC=S^2$
satisfying $\deg(f)>K(f)$; see also~\cite[\S 5]{Bergweiler10} for an 
exposition of some of their results.

An extension of the theory
to quasiregular maps $f\colon S^n\to S^n$, where $n\ge 2$ and $S^n$
is the $n$-sphere, was given in \cite{B12}. 
Here the Julia set consists of all points such that 
$S^n\backslash O^+_f(U)$ has capacity zero for all neighborhoods $U$
and the essential hypothesis is
that the degree exceeds the inner dilatation $K_I(f)$.
It turns out that for $n=2$ these two definitions yield the same set;
cf.\ the remark at the end of section~\ref{proof1}, as well
as~\cite[Section~6]{BergweilerNicks}.

A well-known result of Garber \cite{Garber78}
says that the Julia set of a rational function has positive Hausdorff dimension.
This result was extended by Fletcher and Nicks~\cite{Fletcher11a}
to uniformly quasiregular maps $f\colon S^n\to S^n$. These are, by 
definition, maps such that all iterates are $K$-quasiregular for some
common~$K$.
Also, 
if a quasiregular map
$f\colon S^n\to S^n$ with $\deg(f)>K_I(f)$ is Lipschitz continuous, then
$J(f)$ has positive Hausdorff dimension \cite[Theorem~1.7]{B12}.

The main purpose of this note is to show that 
Garber's result does not extend to the quasiregular setting
without additional hypotheses like uniform
quasiregularity or Lipschitz continuity. We
will actually estimate the Hausdorff measure of the Julia set with respect
to certain gauge functions. We introduce this concept only briefly and
refer to Falconer's book~\cite{Falconer97} for more 
details. 
For $\varepsilon>0$
a continuous,
non-decreasing function $h\colon (0,\varepsilon]\to (0,\infty)$ satisfying
$\lim_{t\to 0}h(t)=0$
is called a gauge function (or dimension function).
The (Euclidean) diameter of a subset $X$ of $\R^n$
is denoted by $\diam X$.
The Hausdorff measure $H_h(A)$ is then defined by
\[
H_h(A)=\lim_{\delta\to 0} \ \inf_{(A_i)}\left\{\sum^\infty_{i=1}h(\diam A_i)\colon
\bigcup^\infty_{i=1}A_i\supset A, \diam (A_i)<\delta\right\}.
\]
It was shown in \cite[Theorem~1.8]{B12} that if $f\colon S^n\to S^n$ is
a quasiregular map 
satisfying $\deg(f)>K_I(f)$ such that the branch set does not
intersect the Julia set, then 
$J(f)$ has positive capacity;
 see section~\ref{prelim} for the definition of the branch set.
In the proof it was actually shown that
\begin{equation}\label{1z}
H_h(J(f))>0
\quad\text{for}\quad
h(t)=\left(\log \frac{1}{t}\right)^{\displaystyle\tfrac{(1-n)\log \deg(f)}{\log K_I(f)}}.
\end{equation}
A result of Wallin \cite{Wallin77} implies that then
$J(f)$ has positive capacity.

First we show that in dimension~$2$,
which is the case considered by Sun and Yang,
the conclusion~\eqref{1z} holds without an additional
hypothesis on the branch set.
Note that in the $2$-dimensional case
the branch set of a quasiregular map is discrete.
As $\CC$ is compact, the branch set of a quasiregular map $f\colon \CC\to\CC$
is actually finite. This simplifies  certain aspects considerably;
cf.~\cite[Section~6]{BergweilerNicks}.
\begin{theorem} \label{thm1}
Let $f\colon \CC\to \CC$ be a quasiregular map satisfying $\deg(f)>K(f)$.
If $\xi\in\CC\backslash E(f)$, then
\begin{equation}\label{1y}
H_h\!\left(\overline{O^-_f(\xi)}\right)>0
\quad\text{for}\quad
h(t)=\left(\log \frac{1}{t}\right)^{\displaystyle -\tfrac{\log \deg(f)}{\log K(f)}}.
\end{equation}
In particular, $H_h(J(f))>0$. Moreover,
$\overline{O^-_f(\xi)}$ and $J(f)$ have positive capacity. 
\end{theorem}

A quasiregular map $f\colon \R^n\to \R^n$ is said to be of polynomial type if
\[
\lim_{x\to \infty}|f(x)|=\infty.
\]
Identifying $S^n$ with $\R^n\cup\{\infty\}$ by stereographic projection, 
a quasiregular map $f\colon \R^n\to \R^n$ of polynomial type
extends to a quasiregular self-map of $S^n$ by putting $f(\infty)=\infty$.
In particular, quasiregular maps $f\colon \C\to \C$ of polynomial type
extend to quasiregular self-maps of $\CC$.
Fletcher and Nicks~\cite{Fletcher11}
have studied the dynamics of quasiregular self-maps of $\R^n$
of polynomial type and shown
that if the degree exceeds the inner dilatation, then $\infty$ is an attracting
fixed point and the boundary of its attracting basin has many properties usually
associated with Julia sets.

Here we only note that for such maps the Julia set is contained in the set
\[BO(f)=\{x\in \R^n\colon (f^k(x))\mbox{ is bounded}\}\]
of points with bounded orbits. (In complex dynamics this set is called the
filled Julia set and usually denoted by $K(f)$, but we reserve the notation
$K(f)$ for the dilatation.) 
We show that the estimate in Theorem~\ref{thm1} is sharp.

\begin{theorem} \label{thm2}
For all $K\in (1,2)$ 
there exists a quasiregular map $f\colon \C\to \C$ of polynomial type
with $\deg(f)=2$
and $K(f)= K$ such that 
\begin{equation}\label{1a}
H_h(J(f))\le H_h(BO(f))<\infty
\quad\text{for}\quad
h(t)=\left(\log\frac{1}{t}\right)^{-\displaystyle \tfrac{\log 2}{\log K}} .
\end{equation}
In particular, $J(f)$ and $BO(f)$ have Hausdorff dimension~$0$.
\end{theorem}
With some more effort one could obtain analogous examples of any given degree. For degrees of the form $2^k$ with $k\in \N$ we only have to replace $f$ by $f^k$.

\section{Preliminaries for the proof of Theorem~\ref{thm1}}
\label{prelim}
We denote the open disk of radius $r$ around a point $a\in\C$
by $D(a,r)$ and the closed disk by $\overline{D}(a,r)$.
The same notation will be used for balls in $\R^n$.
The disk  around $a\in\CC$
with respect to the chordal metric $\chi$ is denoted  by $D_\chi(a,r)$
and the diameter of a subset $A$ of $\CC$ with respect to 
$\chi$ is denoted by $\diam_\chi A$.

An important tool to obtain 
lower bounds for the Hausdorff measure and the Hausdorff dimension
is the mass distribution principle.
We will use the following version;
see~\cite[Theorem~7.6.1]{Przytycki10}.
\begin{lemma} \label{lemma1}
Let $A\subset \R^n$ be compact
and let $h$ be a gauge function.
Suppose that there exist a
probability measure $\mu$ supported on $A$
and a positive constant $C$
such that
$\mu(D(x,r))\leq C\,h(r)$ for $0<r\leq\varepsilon$ and all $x\in A$.
Then $H_h(A) >0$.
\end{lemma}
For the definition and basic properties of quasiregular maps 
we refer to Rickman's book~\cite{Rickman93}.
A standard book for the the $2$-dimensional case is the book
by Lehto and Virtanen~\cite{Lehto73}. Note that their book, except
for the last chapter, deals with quasiconformal maps, i.e., injective
quasiregular maps. However, since every quasiregular map can be 
written as the composition of an analytic map with a quasiconformal
one (cf.~\cite[Chapter VI]{Lehto73}),
many properties of quasiconformal maps extend to quasiregular ones.

Let $\Omega$ be a domain in $\R^n$ and let $f\colon\Omega\to\R^n$ be
a (non-constant) quasi\-regular map.
The \emph{local index} $i(x,f)$ at a point $x\in\Omega$ is defined by
\[
i(x,f) =\inf_U \sup_{y\in\R^n}\card\!\left( f^{-1}(y)\cap U\right),
\]
where the infimum is taken over all neighborhoods $U\subset \Omega$ of~$x$.
Thus $i(x,f)=1$ if and only if $f$ is injective in a neighborhood of~$x$.
The \emph{branch set} 
consists of all $x\in \Omega$ for which $i(x,f)\geq 2$.

As already mentioned, the $2$-dimensional case (i.e.\ the case $n=2$)
is somewhat easier
to deal with since then the branch set is a discrete subset of $\Omega$.
Its elements are called critical points.
For a critical point $c$ we call $i(x,f)-1$ the \emph{multiplicity} of~$c$.
An important tool is the following result
known as the Riemann-Hurwitz Formula; see \cite[\S 5.4]{Beardon91},
\cite[p.~68]{Milnor06} or \cite[\S 1.3]{Steinmetz93}.
Here $\chi(\Omega)$ denotes the Euler characteristic of a domain~$\Omega$.
\begin{lemma}\label{lemma2} 
Let $\Omega_1$ and $\Omega_2$ be domains in $\CC$
and let $f\colon \Omega_1\to \Omega_2$ be a proper quasi\-regulars
 map of degree~$d$.
Denote by $s$ the number of critical points of~$f$, counting multiplicity; that is,
\[s=\sum_{x\in B_f}(i(x,f)-1).\]
Then
\begin{equation}
\label{6b} \chi(\Omega_1)+s=d\chi(\Omega_2).
\end{equation}
\end{lemma}
Since $\chi(\CC)=2$, the equation \eqref{6b} takes the form $s=2d-2$
if $\Omega_1=\Omega_2=\CC$.
Thus, counting multiplicities, the number of critical points of 
a quasiregular map $f\colon \CC\to\CC$ is equal to $2\deg(f)-2$,
as in the case of rational functions.

If $\Omega_j$ is a domain of connectivity $c_j$, then $\chi(\Omega_j)=2-c_j$
and \eqref{6b} takes the form 
\begin{equation}\label{6a} c_1-2=d(c_2-2)+r.
\end{equation}
We shall only need the case that $c_1=c_2=1$.
Then~\eqref{6a} simplifies to 
\begin{equation}\label{6c} s=d-1.
\end{equation}

A consequence is the following result.
\begin{lemma}\label{lemma3a} 
Let $f\colon\CC\to\CC$ be a non-constant quasiregular map
and $V\subset \CC$ a simply connected domain.
Denote by $n$ the number of components of $f^{-1}(U)$
and by $s$ the number of critical points in  $f^{-1}(U)$,
counting multiplicities.
If all components of $f^{-1}(U)$ are simply connected, then $n= \deg(f)-s$.
\end{lemma}
\begin{proof}
Denote by $V_1,\dots,V_n$ the components of $f^{-1}(U)$, by
$s_j$ the number of critical points in $V_j$ and by 
$d_j$ the degree of the proper map $f\colon V_j\to U$.
Then 
\[
\sum_{j=1}^n s_j=s
\quad\text{and}\quad
\sum_{j=1}^n d_j=\deg(f).
\]
By~\eqref{6c} we have $s_j=d_j-1$.
Hence 
\[
n=\sum_{j=1}^n (d_j-s_j)= \sum_{j=1}^n d_j-\sum_{j=1}^n s_j =\deg(f)-s.
\qedhere
\]
\end{proof}

\begin{lemma}\label{lemma3} 
Let $f\colon\CC\to\CC$ be quasiregular. Then there exists $\eta>0$ such 
that if $U\subset\CC$ is a simply connected domain satisfying
$\diam_\chi U< \eta$, then all components of $f^{-1}(U)$ are simply connected.
\end{lemma}
\begin{proof}
As $f$ is continuous, there exists $\delta >0$ such that
$\chi(f(z),f(w))<1$ for $z,w\in\CC$ with $\chi(z,w)<1$.
We may assume that $f$ is non-constant.
This implies that there exists $\eta>0$ such that if
$z\in\CC$, then all components of $f^{-1}(D_\chi(z,\eta))$ 
have chordal diameter less than $\delta$. We may assume that $\eta<1$.

Let now $U\subset \CC$ be a simply connected domain satisfying
$\diam_\chi U< \eta$ and let $V$ be a component of $f^{-1}(U)$.
Then $\diam_\chi V< \delta$.
Since $f\colon V\to U$ is proper, we have $f(\partial V)=\partial U$.
Suppose that $V$ is  multiply connected.
Then $V$ contains a Jordan curve $\gamma$ such that both 
complementary components of $\gamma$ intersect $\partial V$.
Since $\diam_\chi \gamma\leq \diam_\chi V<\delta$, one of these
two complementary components of $\gamma$ has chordal diameter less than~$\delta$.
Denote this component by~$W$; that is, $W$ is a component of 
$\CC\backslash\gamma$ with $\diam_\chi W<\delta$.
Then $\diam_\chi f(W)<1$ by the choice of $\delta$. Moreover, 
$W\cap \partial V\neq\emptyset$ and $\partial W=\gamma\subset V$.
This implies that $f(W)\cap\partial U=f(W\cap\partial V)\neq\emptyset$ and
$\partial f(W)\subset f(\partial W)\subset f(V)=U$.
We deduce that $f(W)\supset\CC\backslash U$. Since $\diam_\chi U<\eta<1$,
but also $\diam_\chi f(W)<1$, this is a contradiction.
\end{proof}

The following estimate is far from sharp, but suffices for our purposes.
\begin{lemma}\label{lemma4} 
Let $f\colon\CC\to\CC$ be a quasiregular map of degree at least~$2$ and
let  $z\in \CC\backslash E(f)$. Then $\card f^{-6}(z)\geq 3$.
\end{lemma}
\begin{proof}
Suppose that $a\in\CC$ satisfies $\card f^{-1}(a)=1$, say $f^{-1}(a)=\{b\}$.
Then $i(x,b)=\deg(f)$. Thus $b$ is a critical point of multiplicity $\deg(f)-1$,
which is the maximal multiplicity a critical point can have.
As the number of critical points, counting multiplicities, is 
equal to $2\deg(f)-2$, 
there are at most two critical points of this maximal multiplicity, and
thus at most two such values of~$a$.

Suppose now that $\card f^{-6}(z)\leq 2$. Then at least five of
the six sets $f^{-k}(z)$, where $k\in\{1,\dots,6\}$, consist only of
critical points of maximal multiplicity.
Thus some $f^{-k}(z)$ contains a critical point $b$ of maximal multiplicity
such that $f(b)$ and $f^2(b)$ are also critical points of maximal
multiplicity. As there are at most two such critical points, we see that
the union of the sets $f^{-k}(z)$ contains a periodic orbit consisting
only of critical points of maximal multiplicity. Hence $z$ is also
in this orbit, and $O^-(z)$ is equal to this orbit, contradicting the
assumption that $z\in E(f)$.
\end{proof}

\section{Proof of Theorem~\ref{thm1}} \label{proof1}
We put $K=K(f)$ and $d=\deg(f)$.
We may assume that $K>1$ since otherwise $f$ is a rational 
function so that the conclusion follows from the result of
Garber already mentioned in the introduction.
We note that a  $K$-quasiregular map is H\"older continuous with exponent
$\alpha=1/K$; see, e.g., \cite[\S II.3.4]{Lehto73}.
Thus there exists $M>0$ such that $\chi(f(z),f(w))\leq M \chi(z,w)^\alpha$
for all $z,w\in\CC$. Induction shows that
\[
\chi(f^k(z),f^k(w))\leq M^{1+\alpha+\dots+\alpha^{k-1}} \chi(z,w)^{\alpha^k}
\]
for $k\in\N$ and $z,w\in\CC$. With $L=M^{1/(1-\alpha)}$ we thus have
\begin{equation}\label{5a}
\chi(f^k(z),f^k(w))\leq L \chi(z,w)^{\alpha^k}
\end{equation}
for $k\in\N$ and $z,w\in\CC$.

Choose $\eta$ according
to Lemma~\ref{lemma3} and let $0<\varepsilon<\min\{L,\eta/2\}$.
For $w\in\CC\backslash E(f)$ we will inductively
define a sequence $(N_m(w))_{m\geq 0}$ of positive integers and,
for $j\in\{1,\dots,N_m(w)\}$, we will also define domains
$U_{m,j}(w)$ and $V_{m,j}(w)$ and points $a_{m,j}(w)$ satisfying
$a_{m,j}(w)\in U_{m,j}(w)\subset V_{m,j}(w)$.
First we put $N_0(w)=1$, $U_{0,1}(w)=V_{0,1}(w)=D_\chi(w,\varepsilon)$ and
$a_{0,1}(w)=w$.
Assuming that $N_{m-1}(w)$, the domains $U_{m-1,j}(w)$ and $V_{m-1,j}(w)$ 
and the points $a_{m-1,j}(w)$ have been defined, we 
define $N_m(w)$ as the number of components of 
$$f^{-1}\left(\bigcup_{i=1}^{N_{m-1}(w)} U_{m-1,i}(w)\right)$$
and we denote these components by $V_{m,1}(w),\dots,V_{m,N_m(w)}(w)$.
Then we choose 
\[
a_{m,j}(w)\in
V_{m,j}(w)\cap f^{-1}\!\left(
\left\{ a_{m-1,i}(w)\colon 1\leq i\leq N_{m-1}(w)\right\}\right)
\]
and we define $U_{m,j}(w)$ as the component of
$V_{m,j}(w)\cap D_\chi(a_{m,j}(w),\varepsilon)$ that contains $a_{m,j}(w)$.
It follows from Lemma~\ref{lemma3} and the choice of $\eta$ and
$\varepsilon$ that the $V_{m,j}(w)$ and hence the $U_{m,j}(w)$ are
simply connected.

If $z\in\partial U_{m,j}(w)$, then $\chi(f^l(z),f^l(a_{m,j}(w)))=\varepsilon$
for some $l$ satisfying $0\leq l\leq m$.
Hence 
\[
\chi(z,a_{m,j})\geq \left(\frac{\chi(f^l(z),f^l(a_{m,j}(w)))}{L}\right)^{1/\alpha^l}
= \left(\frac{\varepsilon}{L}\right)^{K^l} \geq \left(\frac{\varepsilon}{L}\right)^{K^m}
\]
for $z\in\partial U_{m,j}(w)$ by~\eqref{5a}.
With $r_m=(\varepsilon/L)^{K^m}$ we thus find that
\begin{equation}\label{5c}
D_\chi(a_{m,j}(w),r_m)\subset U_{m,j}(w)
\end{equation}
for $w\in \CC\backslash E(f)$, $m\geq 0$ and $1\leq j\leq N_m(w)$.

We now fix a point $\xi\in\CC\backslash E(f)$ and put $N_m=N_m(\xi)$,
$U_{m,j}=U_{m,j}(\xi)$ and $a_{m,j}=a_{m,j}(\xi)$.
Let $s_{m,j}$ be the number of critical points in $f^{-1}(U_{m,j})$ and let
$n_{m,j}$ be the number of components of $f^{-1}(U_{m,j})$.
Then $n_{m,j}=d-s_{m,j}$ by Lemma~\ref{lemma3a}.
Thus 
\[
N_{m+1}=\sum_{j=1}^{N_m} n_{m,j}
=\sum_{j=1}^{N_m} (d-s_{m,j}) =dN_m -\sum_{j=1}^{N_m} s_{m,j}
\geq dN_m-(2d-2).
\]
Writing this inequality in the form $N_{m+1}-2\geq d(N_m-2)$
we see by induction that
\[
N_{m+l}-2\geq d^l(N_m-2)
\]
for $l\in\N$.

By Lemma~\ref{lemma4} we have $\card f^{-6}(\xi)\geq 3$.
Choosing $\varepsilon$ sufficiently small we may thus achieve that $N_6\geq 3$.
Hence 
\begin{equation}\label{5d}
N_m\geq d^{m-6}(N_6-2)+2\geq d^{m-6}
\end{equation}
for $m\geq 6$. In the opposite direction, we clearly have $N_m\leq d^m$.

For $m\in\N$ we put 
$A_m=\{a_{m,j}\colon 1\leq j\leq N_m\}$ and define a probability 
measure $\mu_m$ on $\CC$ by
\[
\mu_m=\frac{1}{N_m}\sum_{z\in A_m} \delta_z,
\]
where $\delta_z$ denotes the Dirac measure. By \cite[Theorem~6.5]{Walters}, 
the sequence $(\mu_m)$ has a subsequence which
converges with respect to the weak$^*$-topology, say $\mu_{m_j}\to \mu$.
By construction, the supports of the measures $\mu_m$, and hence 
the support of $\mu$, are contained in $\overline{O_f^-(\xi)}$.

In order to apply Lemma~\ref{lemma1}, we shall estimate
$\mu(D_\chi(z,r))$ for $z\in\CC$ and $0<r\leq \varepsilon/2L$.
(While Lemma~\ref{lemma1} is stated in terms of Euclidean balls,
we may also use the chordal metric, as this 
is the restriction of the Euclidean metric in~$\R^3$ to $S^2=\CC$.)
We choose $l\in\N$ such that $r_{l+1}/2<r\leq r_l/2$.
Then 
\begin{equation}\label{5e}
K^{l+1}
\log\frac{L}{\varepsilon} = \log \frac{1}{r_{l+1}} \geq \log \frac{1}{r} -\log 2.
\end{equation}
Suppose that $\mu_{k+l}  ( D_\chi(z,r))\neq 0$ for some $k\in\N$. Then 
$A_{k+l}\cap D_\chi(z,r)\neq \emptyset$.
We choose $a\in A_{k+l}\cap D_\chi(z,r)$ and put $w=f^l(a)$.
Then $w\in A_k$ and $a=a_{l,j}(w)$ for some $j\in\{1,\dots,N_l(w)\}$.
Since $r\leq r_l/2$ we deduce from \eqref{5c} that 
\[
D_\chi(z,r)\subset D_\chi(a,2r)\subset D_\chi(a,r_l) \subset U_{l,j}(w).
\]
This implies that $a$ is the only point in $A_{k+l}\cap D_\chi(z,r)$
which is mapped onto~$w$ by~$f^l$. Hence 
\[
\card\!\left(A_{k+l}\cap D_\chi(z,r)\right) \leq \card A_k =N_k.
\]
Using \eqref{5d} we deduce that
\[
\mu_{k+l}  ( D_\chi(z,r)) \leq \frac{N_k}{N_{k+l}}
\leq \frac{d^k}{d^{k+l-6}} = \frac{d^7}{d^{l+1}} =
d^7 \left( K^{l+1}\right)^{\displaystyle -\tfrac{\log d}{\log K}} .
\]
Using~\eqref{5e} we see that 
\[
\mu_{k+l}  ( D_\chi(z,r)) \leq 
C \left(  \log \frac{1}{r}\right)^{\displaystyle -\tfrac{\log d}{\log K}}
\]
for some constant~$C$. Clearly, the same estimate is also satisfied by
the limit measure $\mu$. Now~\eqref{1y} follows from
Lemma~\ref{lemma1}.
Since $J(f)= \overline{O^-_f(z)}$ for all $z\in J(f)$, we 
also have $H_h(J(f))>0$. Finally the conclusion about the 
capacity of $ \overline{O^-_f(z)}$ and $J(f)$ follows from
the result of Wallin~\cite{Wallin77} already quoted.\qed

\begin{remark}
Let $f\colon\CC\to\CC$ be a quasiregular map with $\deg(f)>K(f)$,
let $U$ be an open set such $\CC\backslash O^+_f(U)$ has capacity zero and
let $\xi\in \CC\backslash E(f)$.
By Theorem~\ref{thm1}, $\overline{O_f^-(\xi)}$ has positive capacity.
Thus $O^+_f(U)\cap \overline{O_f^-(\xi)} \neq\emptyset$.
Since $O_f^+(U)$ is open we actually have
$O^+_f(U)\cap O_f^-(\xi)\neq\emptyset$, which implies that $\xi\in O^+_f(U)$.
We deduce that $\CC\backslash O^+_f(U)\subset E(f)$ 
and hence $\card \CC\backslash O^+_f(U)\leq 2$ 
whenever $\CC\backslash O^+_f(U)$ has capacity zero.
This shows that the two definitions of $J(f)$ mentioned 
in the introduction agree for  $f\colon\CC\to\CC$ with $\deg(f)>K(f)$.
\end{remark}

\section{Proof of Theorem~\ref{thm2}} \label{proof2}
Let $0<\delta<1/14$ and put $\lambda=2e/\delta$.  First we define $f$ in
\[Z=\C\backslash(D(1,\delta)\cup D(-1,\delta)).\]
In order to do so we put $f(z)=\lambda(z^2-1)$ for 
$z\in \C\backslash (D(1,2\delta)\cup D(-1,2\delta))$, put 
$f(z)=\pm 2\lambda(z \mp 1)$ for $z\in  \partial D(\pm 1, \delta)$,
and define $f$ by interpolation in 
the annuli $D(\pm 1, 2\delta)\backslash \overline{D}(\pm 1,\delta)$; 
that is, we put
\begin{align*}
      f(z) & = \frac{|z\mp 1|-\delta}{\delta} \lambda(z^2-1)\pm \frac{2\delta-|z\mp 1|}{\delta}2\lambda(z\mp 1) \\
           & = \lambda\left((z\mp 1)^2 \frac{|z\mp 1|}{\delta}\pm 2(z\mp 1)-(z\mp 1)^2\right)
\end{align*}
if $\delta\le |z\mp 1|\le 2\delta$.  We now compute the dilatation of $f$ in
the annuli $D(\pm 1, 2\delta)\backslash \overline{D}(\pm 1,\delta)$.
For simplicity, we consider only 
$D(1, 2\delta)\backslash \overline{D}(1,\delta)$, the argument for 
$D(-1, 2\delta)\backslash \overline{D}(-1,\delta)$ being analogous. As
\[
\left|\frac{\partial}{\partial z}|z-1|\right|
=\left|\frac{\partial}{\partial z}\sqrt{(z-1)(\overline{z}-1)}\right|
=\left|\frac{\overline{z}-1}{2|z-1|}\right|
=\frac{1}{2}
\]
and also
\[\left|\frac{\partial}{\partial \overline{z}}|z-1|\right|
=\frac{1}{2},\]
we see that if $\delta<|z-1|<2\delta$, then
\[\left|\frac{\partial f}{\partial\overline{z}}\right|
=\frac{\lambda}{2\delta}|z-1|^2 \le 2\lambda \delta\]
while
\begin{align*}
\left|\frac{\partial f}{\partial z}\right| 
     & =\lambda \left| 2(z-1)\frac{|z-1|}{\delta}+\frac{(z-1)^2}{\delta}\frac{\partial}{\partial z}|z-1|+2-2(z-1)\right|\\
     & \ge \lambda\left(2-2|z-1|-2\frac{|z-1|^2}{\delta}-\frac{|z-1|^2}{2\delta}\right) 
\ge \lambda(2-14 \delta).
\end{align*}
Thus
  \[
\left.\left|\frac{\partial f}{\partial\overline{z}}(z)
\right/\frac{\partial f}{\partial z}(z)\right|
\le \frac{2\lambda\delta}{\lambda(2-14\delta)}=\frac{\delta}{1-7\delta}\]
for $\delta<|z-1|<2\delta$. As mentioned, the same argument shows that
the last inequality also holds for $\delta <|z+1|<2\delta$.
Choosing $\delta$ small we can thus achieve that $f$ is $K$-quasiconformal
in $D(\pm 1, 2\delta)\backslash \overline{D}(\pm 1,\delta)$ and hence in the
interior of~$Z$.

Next we claim that
\begin{equation}\label{2a} |f(z)|\ge 4\quad \mbox{ for }z\in Z.\end{equation}
In fact, if $z\in \C\backslash (D(1,2\delta)\cup D(-1, 2\delta))$, then
both terms $|z+1|$ and $|z-1|$ are greater than or equal to $2\delta$ while at
least one of them is greater than or equal to~$1$, so that
\[
|f(z)|= \lambda|(z+1)(z-1)|\ge 2 \lambda \delta=4e\ge 4.
\]
  If $z\in D(1,2\delta)\backslash D(1,\delta)$, then
\begin{align*}
      |f(z)| & \ge \lambda\left(2|z-1|-|z-1|^2-\frac{|z-1|^3}{\delta}\right)
 \\ & 
         = \lambda |z-1|\left(2-|z-1|-\frac{|z-1|^2}{\delta}\right) 
         \ge \lambda \delta (2-6\delta) 
         = 2e(2-6\delta)
         \ge 4,
\end{align*}
and the same estimate holds for $z\in D(-1,2\delta)\backslash D(-1,\delta)$. Thus \eqref{2a} holds.

If $|z|\ge 4$, then
\[
|f(z)|=\lambda |z^2-1|
\geq \lambda (|z|^2-1)\ge \lambda(|z|-1)(|z|+1)\ge 3\lambda|z|\ge 3|z|.
\]
Hence
\begin{equation}\label{2b} 
|f^k(z)|\ge 3^k|z|\quad \mbox{ for }|z|\ge 4 \mbox{ and } k\in \N.
\end{equation}
It follows from \eqref{2a} and \eqref{2b} that $f(Z)\subset Z$ and
\begin{equation}\label{2c} 
f^k(z)\to \infty\quad\mbox{ for } z\in Z
\end{equation}
as $k\to \infty$.

So far we have defined $f$ only in~$Z$. We now extend $f$ to the disks $D(\pm 1,\delta)$. In order to do so, we put $t_0=4e$, $\alpha=\delta/4$ and define a sequence $(t_n)$ by
\[
t_n=t_0\alpha^n \exp\left(-\frac{K^n-1}{K-1}\right).
\]
Note that $t_1=t_0\alpha e^{-1}=\delta$ and
\[
f(\partial D(\pm1,t_1))=f(\partial D(\pm 1,\delta))=\partial D(0,4e)=\partial D(0,t_0).
\]
More precisely,
\begin{equation}\label{2d} 
f(\pm 1+t_1e^{i\varphi})=\pm t_0e^{i\varphi}.
\end{equation}
We also define sequences $(s_n)_{n\ge 0}$ and $(r_n)_{n\ge 1}$ by
\[s_n=t_n \exp(-K^n)\quad\mbox{ and }\quad r_n=\frac{s_{n-1}}{s_0}\]
so that $r_1=1$. For later use we note that
\begin{equation}\label{2d1} 
\frac{t_{n+1}}{t_n}=\alpha \exp\left(-\frac{K^{n+1}-K^n}{K-1}\right)=\alpha \exp (-K^n)
\end{equation}
and thus
\begin{equation}\label{2e} 
t_{n+1}=\alpha t_n \exp(-K^n)=\alpha s_n
\quad\mbox{ and }\quad r_n=\frac{t_n}{\alpha s_0}.
\end{equation}

For $n\in\N$ we put $\Sigma_n=\{-1,1\}^n$. We also put $\Sigma_0=\{\emptyset\}$
and
\[\Sigma= \bigcup^\infty_{n=1}\Sigma_n\quad\mbox{ and }\quad\Sigma'= \bigcup^\infty_{n=0}\Sigma_n=\Sigma \cup\{\emptyset\}.\]
The shift $\sigma\colon \Sigma\to \Sigma'$ is defined by
\[
\sigma(u_1,\ldots,u_n)=(u_2,\cdots,u_n)
\]
for $n\ge 2$. We also define $\tau\colon \Sigma\to \Sigma'$ by 
\[
\tau(u_1,\ldots,u_n) =u_1\sigma(u_1,\ldots, u_n) =(u_1u_2, u_1u_3,\ldots,u_1u_n)
\]
for $n\ge 2$. For $n=1$ we put $\sigma(1)=\tau(1)=\sigma(-1)=\tau(-1)=\emptyset$. Thus $\sigma(\Sigma_n)=\tau(\Sigma_n)=\Sigma_{n-1}$ for all $n\in \N$.

We define $a\colon \Sigma'\to \C$, writing $a_u$ instead of $a(u)$, by $a_\emptyset=0$ and 
\[a_u
=\sum^n_{j=1} u_jr_j\quad\mbox{for }u=(u_1,\ldots,u_n)\in \Sigma_n
\mbox{ where }  n\in \N.\]
For $n\in\N$,
$u=(u_1,\ldots,u_n)\in \Sigma_n$ and $\varepsilon\in \{-1,1\}$ we write 
\[
(u,\varepsilon)=(u_1,\ldots,u_n,\varepsilon)\in \Sigma_{n+1}
\]
and put
\[X(u)=D(a_u,s_n)\backslash (D(a_{(u,1)},t_{n+1})\cup D(a_{(u,-1)},t_{n+1}))\]
and 
\[Y(u)=D(a_u,t_n)\backslash D(a_u,s_n),\]
see Figure~\ref{fig1}.
\begin{figure}[htb] 
\centering
\setlength{\unitlength}{60pt}
\begin{picture}(4.3,4.3)(-2.15,-2.15)
\put(0,0){\circle*{0.03}}
\put(0.03,-0.1){$a_u$}
\put(0,0){\circle{3.6}}
\put(0,0){\circle{4.3}}
\put(-1,0){\circle*{0.03}}
\put(1,0){\circle*{0.03}}
\put(-1,0){\circle{0.6}}
\put(-1,0){\circle{0.9}}
\put(-1.166,0){\circle{0.15}}
\put(-0.833,0){\circle{0.15}}
\put(1,0){\circle{0.6}}
\put(1,0){\circle{0.9}}
\put(1.166,0){\circle{0.15}}
\put(0.833,0){\circle{0.15}}
\put(0,1){$X(u)$}
\put(0,1.88){$Y(u)$}
\put(-2.75,1.35){$Y(u,\!-1)$}
\put(-2.0,1.36){\vector(1,-1){1.0}}
\put(-3.1,0.1){$X(u,\!-1)$}
\put(-2.3,0.14){\vector(1,0){1.3}}
\put(-2.3,-1.3){$a_{(u,\!-1)}$}
\put(-2.16,-1.16){\vector(1,1){1.13}}
\put(2.05,1.35){$Y(u,\!1)$}
\put(2.0,1.36){\vector(-1,-1){1.0}}
\put(2.3,0.1){$X(u,\!1)$}\put(2.3,0.14){\vector(-1,0){1.3}}
\put(-2.3,-1.3){$a_{(u,\!-1)}$}
\put(2.2,-1.3){$a_{(u,1)}$}
\put(2.16,-1.16){\vector(-1,1){1.13}}
\end{picture}
\caption{$X(u)$ and $Y(u)$, not to scale: the actual annuli $Y(u)$ are much
wider and the $X(u,\pm 1)$ and $Y(u,\pm 1)$ are much smaller in comparison 
to $X(u)$ and $Y(u)$ than shown.}
\label{fig1}
\end{figure}
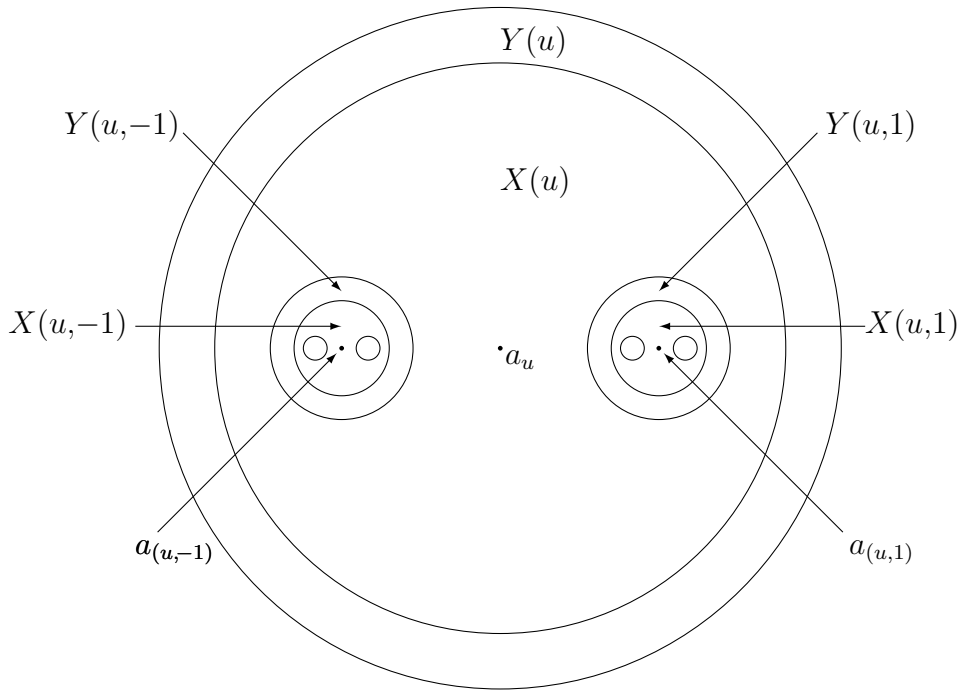

Let
\[
X=\bigcup_{u\in \Sigma}X(u)\quad\mbox{and}\quad Y=\bigcup_{u\in \Sigma} Y(u).
\]
We extend $f$ to $X\cup Y$ in such a way that $f$ maps $X(u)$ to $X(\tau(u))$
and $Y(u)$ to $Y(\tau(u))$ for all $u\in \Sigma$. Moreover, $f$ is an affine
function on
each $X(u)$ and $f$ is a suitably rescaled version of the function
$z\mapsto z|z|^{1/K-1}$ on each $Y(u)$, implying that $f$ is conformal on
$X$ and quasiconformal on $Y$ with $K(f)=K$. More precisely, we put 
\begin{equation}\label{2f} f(z)
=a_{\tau(u)}+u_1\frac{s_{n-1}}{s_n}(z-a_u)\quad \mbox{for }z\in X(u)
\end{equation}
and
\begin{equation}\label{2g} f(z)=a_{\tau(u)}+u_1\frac{s_{n-1}}{s_n^{1/K}}|z-a_u|^{1/K-1}(z-a_u)\quad \mbox{for }z\in Y(u).\end{equation}
Note that the expressions given by \eqref{2f} and \eqref{2g} agree for
$z\in \partial X(u)\cap \partial Y(u)$, since for such $z$ we have
$|z-a_u|=s_n$ and hence 
\[
\frac{s_{n-1}}{s_n^{1/K}}|z-a_u|^{1/K-1}=\frac{s_{n-1}}{s_n^{1/K}}s_n^{1/K-1}=\frac{s_{n-1}}{s_n}.
\]
The two expressions for $f$
also agree for $z\in \partial X(u)\cap \partial Y(u,\varepsilon)$ where
$\varepsilon\in \{1,-1\}$. In this case we have 
$|z-a_{(u,\varepsilon)}|=t_{n+1}$.
Thus, by \eqref{2d1} and \eqref{2e}, 
\begin{align*}
\frac{s_{n}}{s_{n+1}^{1/K}} |z-a_{(u,\varepsilon)}|^{1/K-1}
& =\frac{s_{n}}{s_{n+1}^{1/K}}t_{n+1}^{1/K-1} 
 =\frac{s_n}{t_{n+1}}\left(\frac{t_{n+1}}{s_{n+1}}\right)^{1/K}
=\frac{1}{\alpha}\left(\exp(-K^{n+1})\right)^{1/K}
 \\ & 
= \frac{1}{\alpha}\exp(-K^n)
=\frac{t_n}{t_{n+1}}
=\frac{s_{n-1}}{s_n}
=\frac{r_n}{r_{n+1}}
\end{align*}
so that
\begin{align*}
   a_{\tau(u,\varepsilon)}+ u_1\frac{s_n}{s_{n+1}^{1/K}}
             |z-a_{(u,\varepsilon)}|^{1/K-1}(z-a_{(u,\varepsilon)})  
&  =  u_1a_{\sigma(u,\varepsilon)}+ u_1\frac{s_{n-1}}{s_n} (z-a_u-r_{n+1}\varepsilon)  \\
   & =  u_1a_{\sigma(u,\varepsilon)}+ u_1\frac{s_{n-1}}{s_n}(z-a_u)-u_1r_{n}\varepsilon \\
   & =  u_1(a_{\sigma(u,\varepsilon)}-r_n\varepsilon)+ u_1\frac{s_{n-1}}{s_n}(z-a_u) \\
   & =  a_{\tau(u)}+u_1 \frac{s_{n-1}}{s_n}(z-a_u).
\end{align*}
Hence the expressions given by \eqref{2f} and \eqref{2g} agree for 
$z\in \partial X(u)\cap \partial Y(u,\varepsilon)$.
    
We deduce that $f$ is continuous and  in fact quasiregular
with $K(f)=K$ on $X\cup Y$.
Moreover, we can deduce from \eqref{2d} and \eqref{2g} that $f$ is 
continuous on 
\[
\partial D(1,t_1)\cup \partial D(-1,t_1)=\partial Y\cap \partial Z.
\]
Thus $f$ is continuous and quasiregular on  $X\cup Y\cup Z$,
which is the set where $f$ has been defined so far.
    
In order to define $f$ on $\C\backslash (X\cup Y\cup Z)$ we denote by
$\Sigma_\infty$ the set of all sequences $(u_k)_{k\in \N}$ with
$u_k\in \{-1,1\}$ for all $k\in \N$. The shift
$\sigma\colon \Sigma_\infty\to \Sigma_\infty$ and the map
$\tau\colon \Sigma_\infty\to \Sigma_\infty$ are defined as before; that is,
\[
\sigma((u_k)_{k\in \N})=(u_{k+1})_{k\in \N} \quad
\mbox{ and } \quad
\tau((u_k)_{k\in \N})=u_1\sigma((u_k)_{k\in \N}).
\]
For $u=(u_k)_{k\in \N}\in \Sigma_\infty$ we put
$a_u=\sum^\infty_{k=1}u_kr_k$. Moreover, we put
\[
C=\{a_u\colon u\in \Sigma_\infty\}. 
\]
Noting that $\diam X(u)=2s_n$
for $u\in \Sigma_n$ and $s_n\to 0$ as $n\to \infty$, we easily see that
$C=\C\backslash (X\cup Y\cup Z)$ and that $f$ extends continuously and in fact
quasiregularly to $\C$ by putting $f(a_u)=a_{\tau(u)}$ for $u\in \Sigma_\infty$.
Moreover, the extended map satisfies $K(f)=K$.

We mention that the existence of a quasiregular extension of $f$ from 
the domain
$X\cup Y\cup Z=\C\backslash C$ to
$\C$ also follows from a general removability result for quasiregular 
maps~\cite[Corollary~ 1.5]{Astala94}, together with the assertion that 
$\dim C=0$ proved below.

For $u\in \Sigma_n$ we have $f(X(u))=X(\tau(u))$ and $f(Y(u))=Y(\tau(u))$
so that 
\[
f^n(X(u))=X(\emptyset)=D(0,s_0)\backslash(D(1,t_1)\cup  D(-1,t_1))
\]
and
\[
f^n(Y(u))=Y(\emptyset)=D(0,t_0)\backslash D(0,s_0) .
\]
Hence $f^k(z)\to \infty$ for $z\in X\cup Y$ by \eqref{2c}. On the other hand, $f(C)=C$. We deduce that $BO(f)=C$.
    
In order to estimate the Hausdorff measure of $C$ we note that 
\[C\subset \bigcup_{u\in \Sigma_n}D(a_u,s_n)\]
for all $n\in \N$. 
For the function $h$ defined by \eqref{1a} we thus have
\begin{align*}
        \sum_{u\in \Sigma_n} h(\diam D(a_u,s_n)) & = 2^n h(2s_n) 
  = 2^n h\left(\frac{2t_{n+1}}{\alpha}\right) 
  = 2^n\left(\log \frac{\alpha}{2 t_{n+1}}\right)^{\displaystyle -\tfrac{\log 2}{\log K}} \\
& = 2^n\left(-\log 2t_0-n\log \alpha 
+\frac{K^{n+1}-1}{K-1}\right)^{\displaystyle -\tfrac{\log 2}{\log K}} \\
& = \left(K^{-n}\left(-\log 2 t_0-n\log \alpha\right)
+\frac{K-K^{-n}}{K-1}\right)^{\displaystyle -\tfrac{\log 2}{\log K}}. 
\end{align*}
We deduce that 
\[
H_h(C)\le \left(\frac{K}{K-1}\right)^{\displaystyle -\tfrac{\log 2}{\log K}}<\infty,
\]
from which the conclusion follows since $C=BO(f)$.\qed

\end{document}